\documentclass{amsart}

\usepackage[T1]{fontenc}
\usepackage{
amsmath, amsfonts, amssymb, amsthm, mathrsfs, wasysym, graphics, graphicx, xcolor, url, hyperref, hypcap, xargs, multicol, pdflscape, multirow, hvfloat, array, ae, aecompl, pifont, mathtools, a4wide, float, blkarray, overpic, nicefrac}
\usepackage[shortlabels, inline]{enumitem}
\usepackage{bbm}
\usepackage[noabbrev,capitalise]{cleveref}
\usepackage[normalem]{ulem}
\usepackage{marginnote}
\hypersetup{colorlinks=true, citecolor=darkblue, linkcolor=darkblue}
\usepackage[all]{xy}
\usepackage{changepage}
\usepackage{tikz}
\usepackage{tikz-cd}
\usetikzlibrary{trees, decorations, decorations.pathmorphing, decorations.markings, decorations.shapes, shapes, arrows, matrix, calc, fit, intersections, patterns, angles}

\graphicspath{{figures/}{figures/diagonals/}{figures/walks/}{figures/tubes/}{figures/blocks/}}
\makeatletter\def\input@path{{figures/}}\makeatother
\usepackage{caption}
\usepackage[export]{adjustbox}

\newtheorem{theorem}{Theorem}[section]
\newtheorem{corollary}[theorem]{Corollary}
\newtheorem{proposition}[theorem]{Proposition}
\newtheorem{lemma}[theorem]{Lemma}

\newtheorem*{theorem*}{Theorem}

\theoremstyle{definition}

\newtheorem{example}[theorem]{Example}

\crefname{equation}{Equation}{Equations}

\newcommand{\R}{\mathbb{R}} 
\renewcommand{\b}[1]{{\boldsymbol{#1}}} 

\newcommand{\set}[2]{\left\{ #1 \;\middle|\; #2 \right\}} 
\newcommand{\bigset}[2]{\big\{ #1 \;\big|\; #2 \big\}} 
\newcommand{\Bigset}[2]{\Big\{ #1 \;\Big|\; #2 \Big\}} 
\newcommand{\ssm}{\smallsetminus} 
\newcommand{\dotprod}[2]{\langle \, #1 \; | \; #2 \, \rangle} 
\newcommand{\one}{\mathbbm{1}} 
\newcommandx{\ones}[1][1=n]{\one_{#1}} 
\newcommand{\eqdef}{\mbox{\,\raisebox{0.2ex}{\scriptsize\ensuremath{\mathrm:}}\ensuremath{=}\,}} 
\newcommand{\simplex}{\triangle} 

\DeclareMathOperator{\conv}{conv} 
\DeclareMathOperator{\cone}{cone} 

\newcommand{\ie}{\textit{i.e.}~} 
\newcommand{\eg}{\textit{e.g.}~} 
\newcommand{\aka}{\textit{a.k.a.}~} 
\definecolor{darkblue}{rgb}{0,0,0.7} 
\definecolor{green}{RGB}{57,181,74} 
\definecolor{violet}{RGB}{147,39,143} 
\newcommand{\darkblue}{\color{darkblue}} 
\newcommand{\defn}[1]{\textsl{\darkblue #1}} 

\usepackage{todonotes}

\newcommand{\polytope}[1]{\mathsf{#1}} 
\newcommandx{\Perm}[1][1=n]{\polytope{Perm}_{#1}} 

\newcommandx{\Asso}[2][1=n,2={}]{\mathsf{Asso}^{#2}(#1)} 
\newcommandx{\Nest}[2][1=\building,2={}]{\mathsf{Nest}^{#2}(#1)} 

\newcommandx{\Fan}[1][1=F]{\mathcal{#1}} 
\newcommandx{\nestedFan}[1][1=\quiver]{\mathcal{F}(#1)} 

\newcommand{\ivector}[1]{\b{\iota}_{#1}} 

\newcommandx{\ray}[1][1=r]{\b{#1}} 
\newcommandx{\rays}[1][1=R]{\b{#1}} 

\newcommandx{\gZono}[1][1=G]{\mathsf{Z}_{#1}} 
\newcommandx{\gArr}[1][1=G]{\mathcal{A}_{#1}} 
\newcommandx{\gFan}[1][1=G]{\Fan_{#1}} 
\newcommandx{\gFanO}[1][1=G]{\widehat{\Fan}_{#1}} 
\newcommandx{\cc}[1][1=G]{\mathbb{K}_{#1}} 

\newcommandx{\braid}[1][1=n]{\mathcal{B}_{#1}} 
\newcommandx{\sbraid}[1][1=n]{\widehat{\mathcal{B}}_{#1}} 

\newcommandx{\dZono}[1][1=\b{h}]{\mathsf{D}_{#1}} 

\newcommand{\deformationCone}{\mathbb{DC}} 
\newcommandx{\coefficient}[3][1={\b{s}}, 2=\b{r}, 3=\b{r}']{\alpha_{#2,#3}(#1)} 
\newcommandx{\virtualPolytopes}[1][1=d]{\mathbb{V}^{#1}} 
\newcommandx{\VDP}[1][1=n]{\mathbb{VDP}^{#1}} 
\newcommandx{\CVDP}[1][1=n]{\overrightarrow{\mathbb{VDP}}^{#1}} 
\newcommand{\VD}[1][1=n]{\mathbb{VD}} 
\newcommandx{\opcone}[1][1={\mu,\omega}]{\polytope{C}_{#1}}
\newcommandx{\orcone}[1][1={\omega}]{\polytope{C}_{#1}}

\makeatletter
\def\part{\@startsection{part}{1}%
\z@{.7\linespacing\@plus\linespacing}{.8\linespacing}%
{\LARGE\sffamily\centering}}
\makeatother

\makeatletter
\def\l@section{\@tocline{1}{5pt}{0pc}{}{}}
\makeatother
\let\oldtocpart=\tocpart
\renewcommand{\tocpart}[2]{\sc\large\oldtocpart{#1}{#2}}
\let\oldtocsection=\tocsection
\renewcommand{\tocsection}[2]{\bf\oldtocsection{#1}{#2}}
\let\oldtocsubsubsection=\tocsubsubsection
\renewcommand{\tocsubsubsection}[2]{\quad\oldtocsubsubsection{#1}{#2}}

\title{Deformed graphical zonotopes}

\thanks{Partially supported by the French ANR grants CAPPS~17\,CE40\,0018, and CHARMS~19\,CE40\,0017.}

\author{Arnau Padrol}
\address[Arnau Padrol]{Institut de Math\'ematiques de Jussieu - Paris Rive Gauche, Sorbonne Universit\'e, Paris}
\email{arnau.padrol@imj-prg.fr}
\urladdr{\url{https://webusers.imj-prg.fr/~arnau.padrol/}}

\author{Vincent Pilaud}
\address[Vincent Pilaud]{CNRS \& LIX, \'Ecole Polytechnique, Palaiseau}
\email{vincent.pilaud@lix.polytechnique.fr}
\urladdr{\url{http://www.lix.polytechnique.fr/~pilaud/}}

\author{Germain Poullot}
\address[Germain Poullot]{Institut de Math\'ematiques de Jussieu - Paris Rive Gauche, Sorbonne Universit\'e,~Paris}
\email{germain.poullot@imj-prg.fr}
\urladdr{\url{https://webusers.imj-prg.fr/germain.poullot}}

\begin{document}

\begin{abstract}
We study deformations of graphical zonotopes.
Deformations of the classical permutahedron (which is the graphical zonotope of the complete graph) have been intensively studied in recent years under the name of generalized permutahedra.
We provide an irredundant description of the deformation cone of the graphical zonotope associated to a graph~$G$, consisting of independent equations defining its linear span (in terms of non-cliques of~$G$) and of the inequalities defining its facets (in terms of common neighbors of neighbors in~$G$).
In particular, we deduce that the faces of the standard simplex corresponding to induced cliques in~$G$ form a linear basis of the deformation cone, and that the deformation cone is simplicial if and only if $G$ is triangle-free.
\end{abstract}

\vspace*{.4cm}
\maketitle


\section*{Introduction}

The \defn{graphical zonotope} of a graph $G$ is a convex polytope $\gZono$ whose geometry encodes several combinatorial properties of~$G$. 
For example, its vertices are in bijection with the acyclic orientations of~$G$~\mbox{\cite[Prop.~2.5]{Stanley2007}} and its volume is the number of spanning trees of~$G$~\mbox{\cite[Ex.~4.64]{StanleyEC1}}. 
When $G$ is the complete graph~$K_n$, the graphical zonotope is a translation of the classical \mbox{$n$-dimensional} \defn{permutahedron}. This polytope, obtained as the convex hull of the $n!$ permutations of the vector $(1,2,\dots,n)\in \R^n$, was first introduced by Schoute in 1911~\cite{Schoute1911}, and has become one of the most studied polytopes in geometric and algebraic combinatorics.

A \defn{deformed permutahedron} (\aka \defn{generalized permutahedron}) is a polytope obtained from the permutahedron by translating its facet-defining hyperplanes without passing through a vertex. These polytopes were originally introduced by Edmonds in 1970 under the name of \defn{polymatroids} as a polyhedral generalization of matroids in the context of linear optimization~\cite{Edmonds}. They were rediscovered by Postnikov in 2009~\cite{Postnikov}, who initiated the investigation of their rich combinatorial structure. They have since become a widely studied family of polytopes that appears naturally in several areas of mathematics, such as algebraic combinatorics~\cite{AguiarArdila, ArdilaBenedettiDoker, PostnikovReinerWilliams}, optimization~\cite{SubmodularFunctionsOptimization},  game theory~\cite{DanilovKoshevoy2000}, statistics~\cite{MortonPachterShiuSturmfelsWienand2009,MohammadiUhlerWangYu2018}, and economic theory~\cite{JoswigKlimmSpitz2021}. The set of deformed permutahedra can be parametrized by the cone of \defn{submodular functions}~\cite{Edmonds,Postnikov}. 

In general, a deformation of a polytope~$\polytope{P}$ can be equivalently described as 
\begin{enumerate*}[(i)]
 \item a polytope obtained from~$\polytope{P}$ by moving the vertices so that the directions of all edges are preserved~\cite{Postnikov,PostnikovReinerWilliams}, 
 \item a polytope obtained from~$\polytope{P}$ by translating its facet-defining halfspaces without passing through a vertex~\cite{Postnikov,PostnikovReinerWilliams}, 
 \item a polytope whose normal fan coarsens the normal fan of~$\polytope{P}$~\cite{McMullen-typeCone}, 
 \item a polytope whose support function is a convex  piecewise linear continuous function supported on the normal fan of~$\polytope{P}$~\cite[Sec.~6.1]{CoxLittleSchenckToric}\cite[Sec.~9.5]{DeLoeraRambauSantos},  or
 \item a Minkowski summand of a dilate of~$\polytope{P}$~\cite{Shephard, Meyer}.
\end{enumerate*}
The set of deformations of~$\polytope{P}$ always forms a polyhedral cone under dilation and Minkowski addition, which is called the \defn{deformation cone} of~$\polytope{P}$~\cite{Postnikov}.
Its interior is called the \defn{type cone} of the normal fan of $\polytope{P}$~\cite{McMullen-typeCone}, and contains those polytopes with the same normal fan as~$\polytope{P}$. 
When~$\polytope{P}$ has rational vertex coordinates, then the type cone is known as the \defn{numerically effective cone} and encodes the embeddings of the associated toric variety into projective space~\cite{CoxLittleSchenckToric}.

There exist several methods to parametrize and describe the deformation cone of a given polytope (see \eg \cite[App.~15]{PostnikovReinerWilliams}), for example via the \defn{height deformation space} and the \defn{wall-crossing inequalities} or via the \defn{edge deformation space} and the \defn{polygonal face equalities}.
However, these methods only provide redundant inequality descriptions of the deformation cone.
Not even the dimension of the deformation cone is easily deduced from these descriptions, as illustrated by the difficulty of describing which fans have a  nonempty type cone (\ie describing \defn{realizable fans} \cite[Chap.~9.5.3]{DeLoeraRambauSantos}), or a one dimensional type cone (\ie describing \defn{Minkowski indecomposable} polytopes~\cite{Kallay1982,McMullen1987,Meyer,PreslawskiYost2016,Shephard}).

The search of irredundant facet descriptions of deformation cones of particular families of combinatorial polytopes has received considerable attention recently~\cite{ACEP-DeformationsCoxeterPermutahedra, BazierMatteDouvilleMousavandThomasYildirim, CDGRY2020, CastilloLiu2020, PadrolPaluPilaudPlamondon, AlbertinPilaudRitter, PadrolPilaudPoullot-deformedNestohedra}. One of the motivations sparking this interest arises from the \defn{amplituhedron program} to study scattering amplitudes in mathematical physics~\cite{ArkaniHamedTrnka-Amplituhedron}. As described in \cite[Sec.~1.4]{PadrolPaluPilaudPlamondon}, the deformation cone provides canonical realizations of a polytope (seen as a \defn{positive geometry}~\cite{ArkaniHamedBaiLam-PositiveGeometries}) in the positive region of the kinematic space, akin to those of the associahedron in~\cite{ArkaniHamedBaiHeYan}.

The main result of this paper (\cref{thm:main}) presents complete irredundant descriptions of the deformation cones of graphical zonotopes. Note that, since graphical zonotopes are deformed permutahedra, their type cones appear as particular faces of the submodular cone. Faces of the submodular cone are far from being well understood. For example, determining its rays remains an open problem since the 1970s, when it was first asked by Edmonds~\cite{Edmonds}. 

It is worth noting that most of the existing approaches to compute deformation cones only focus on simple polytopes with simplicial normal fans~\cite{ChapotonFominZelevinsky,PostnikovReinerWilliams}. Nevertheless, most graphical zonotopes are not simple. They are simple only for chordful graphs (those where every cycle induces a clique), see \cite[Prop.~5.2]{PostnikovReinerWilliams} and \cite[Rmk.~6.2]{Kim08}.
In this paper, we thus use an alternative approach to describe the deformation cone of a non-simple polytope based on a simplicial refinement of its normal cone.

The paper is organized as follows.
We first recall in \cref{sec:preliminaries} the necessary material concerning polyhedral geometry (\cref{subsec:fansPolytopes}), deformation cones (\cref{subsec:deformationCones}), and graphical zonotopes (\cref{subsec:graphicalZonotopes}).
We then describe in \cref{sec:graphicalTypeCones} the deformation cone of any graphical zonotope, providing first a possibly redundant description (\cref{subsec:firstDescription}), then irredundant descriptions of its linear span (\cref{subsec:linearSpan}) and of its facet-defining inequalities (\cref{subsec:facets}), and finally a characterization of graphical zonotopes with simplicial type cones (\cref{subsec:triangleFree}).


\section{Preliminaries}
\label{sec:preliminaries}


\subsection{Fans and polytopes}
\label{subsec:fansPolytopes}

We mainly follow~\cite{Ziegler-polytopes} for the notation concerning polyhedral geometry, and we refer to it for more background and details.

A \defn{polytope} $\polytope{P}$ in~$\R^d$ is the convex hull of finitely many points. Its \defn{faces} are the zero-sets of non-negative affine functions on~$\polytope{P}$. Its \defn{vertices}, \defn{edges} and \defn{facets} are its faces of dimension~$0$, dimension~$1$, and codimension~$1$, respectively. 
A $d$-dimensional polytope is called \defn{simple} if every vertex is incident to $d$ facets.

Similarly, a \defn{polyhedral cone} $\polytope{C}$ in~$\R^d$ is the positive span of finitely many vectors,
and its \defn{faces} are the zero-sets of non-negative linear functions on~$\polytope{C}$. Its \defn{rays} and \defn{facets} are its faces of dimension and codimension~$1$, respectively. Its \defn{lineality space} is the inclusion-minimal face, which is always the largest linear subspace contained in~$\polytope{C}$.
A cone is \defn{simplicial} if its rays are linearly independent, and \defn{pointed} if its lineality is $\{\b{0}\}$. Note that every cone can be decomposed as the free sum of its lineality with a pointed cone (obtained from any section transversal to the lineality).

A \defn{fan} $\Fan$ in $\R^d$ is a collection of cones closed under taking faces and such that the intersection of any two cones is a common face of the two cones.
Two cones of a fan are \defn{adjacent} if they share a facet. 
The fan $\Fan$ is \defn{complete} if the union of its cones is~$\R^d$, \defn{essential} if all its cones are pointed, and \defn{simplicial} if all its cones are simplicial.
We will say that~$\Fan$ is \defn{supported} on the set of vectors~$\b{S}$ if every cone of~$\Fan$ is the cone spanned by a subset of~$\b{S}$.
An essential fan is supported by representatives of its rays, and this is the unique inclusion-minimal set with this property, up to positive rescaling. 
For non-essential fans, however, non-canonical choices have to be made.
We say that~$\Fan$ \defn{coarsens}~$\Fan[G]$ and that~$\Fan[G]$ \defn{refines}~$\Fan$ if every cone of a fan~$\Fan$ is a union of cones of a fan~$\Fan[G]$.

The \defn{normal cone} of a face~$\polytope{F}$ of a polytope~$\polytope{P}$ in $\R^d$ is the polyhedral cone in the dual space $(\R^d)^*$ (which we identify with $\R^d$ via the standard inner product) consisting of the linear forms whose maximal value on~$\polytope{P}$ is attained on all the points of~$\polytope{F}$. The \defn{normal fan} of $\polytope{P}$ is the collection of all the normal cones to its faces. It is always complete, and essential whenever~$\polytope{P}$ is full dimensional.

The \defn{Minkowski sum} of two polytopes~$\polytope{P}$ and~$\polytope{Q}$ is the polytope $\polytope{P} + \polytope{Q} \eqdef \set{\b{p} + \b{q}}{\b{p} \in \polytope{P} \text{ and } \b{q} \in \polytope{Q}}$. The normal fan of $\polytope{P} + \polytope{Q}$ is the common refinement of the normal fans of~$\polytope{P}$ and~$\polytope{Q}$. 
We say that~$\polytope{P}$ is a \defn{Minkowski summand} of~$\polytope{R}$ if there is a polytope~$\polytope{Q}$ such that $\polytope{P} + \polytope{Q} = \polytope{R}$, and a \defn{weak Minkowski summand} if there is a scalar $\lambda \ge 0$ and a polytope~$\polytope{Q}$ such that ${\polytope{P} + \polytope{Q} = \lambda \polytope{R}}$. Equivalently, $\polytope{P}$ is a weak Minkowski summand of~$\polytope{Q}$ if and only if the normal fan of~$\polytope{Q}$ refines the normal fan of~$\polytope{P}$~\cite[Thm.~4]{Shephard}. The polytope $\polytope{P} \subset \R^d$ is called \defn{Minkowski indecomposable} if all its weak Minkowski summands are of the form $\lambda \polytope{P} + \b{t}$ for some scalar~$\lambda \ge 0$ and vector~$\b{t} \in \R^d$.

A \defn{zonotope} is a Minkowski sum of line segments, called its~\defn{generators}. Its normal fan is the fan induced by the arrangement of hyperplanes orthogonal to these segments, see~\cite[Sec.~7.3]{Ziegler-polytopes}.


\subsection{Deformation cones}
\label{subsec:deformationCones}

The weak Minkowski summands of a polytope~$\polytope{P} \subset \R^d$ are also known as \defn{deformations} of~$\polytope{P}$, as they can be always obtained from~$\polytope{P}$ by translations of its facet-defining inequalities. 
It is sometimes convenient to consider the set of deformations of~$\polytope{P}$ embedded inside the real vector space of \defn{virtual $d$-dimensional polytopes}~$\virtualPolytopes$~\cite{PukhlikovKhovanskii}.
This is the set of formal differences of polytopes $\polytope{P} - \polytope{Q}$ under the equivalence relation $(\polytope{P}_1 - \polytope{Q}_1) = (\polytope{P}_2 - \polytope{Q}_2)$ whenever $\polytope{P}_1 + \polytope{P}_2 = \polytope{Q}_1 + \polytope{Q}_2$.
Endowed with Minkowski addition, it is the Grothendieck group of the semigroup of polytopes, which are embedded into~$\virtualPolytopes$ via the map~$\polytope{P} \mapsto \polytope{P} - \{\b{0}\}$.
It extends to a real vector space via dilation: for $\polytope{P} - \polytope{Q} \in \virtualPolytopes$ and $\lambda \in \R$, we set $\lambda (\polytope{P} - \polytope{Q}) \eqdef \lambda \polytope{P} - \lambda \polytope{Q}$ when $\lambda \ge 0$, and $\lambda (\polytope{P} - \polytope{Q}) \eqdef ((-\lambda) \polytope{Q}) - ((-\lambda) \polytope{P})$ when~$\lambda < 0$.
Here, $\lambda \polytope{P} \eqdef \set{\lambda \b{p}}{\b{p} \in \polytope{P}}$ denotes the dilation of $\polytope{P}$ by $\lambda \ge 0$.
(Note in particular that $-\polytope{P}$ does not represent the reflection of $\polytope{P}$, but its group inverse.)

As we already mentioned, the set of deformations of a polytope~$\polytope{P} \subset \R^d$ forms a polyhedral cone under dilation and Minkowski addition, called the \defn{deformation cone} and denoted by~$\deformationCone(\polytope{P})$:
\begin{align*}
\deformationCone(\polytope{P}) & \eqdef \set{\polytope{Q} \subset \R^d}{\polytope{Q} \text{ is a weak Minkowski summand of } \polytope{P}}.
\end{align*}
Note that~$\deformationCone(\polytope{P})$ is a closed convex cone (dilations and Minkowski sums preserve weak Minkowski summands) and contains a lineality subspace of dimension~$d$ (translations preserve weak Minkowski summands). It is the set of all polytopes whose normal fan coarsens the normal fan~$\Fan$ of~$\polytope{P}$.
Its interior consists of all polytopes whose normal fan is~$\Fan$, and was called the \defn{type cone} of~$\Fan$ by McMullen~\cite{McMullen-typeCone}. 
The faces of $\deformationCone(\polytope{P})$ are the deformation cones of the Minkowski summands of~$P$, and the face lattice is
described by the inclusions $\deformationCone(\polytope{Q}) \subseteq \deformationCone(\polytope{R})$ whenever the normal fan of~$\polytope{Q}$ coarsens the normal fan of~$\polytope{R}$.
Having into account the lineality, we will say that the deformation cone is \defn{simplicial} when its quotient modulo translations is simplicial. We will also talk about the rays of~$\deformationCone(\polytope{P})$, meaning the rays of its quotient modulo translations.
They are spanned by the indecomposable Minkowski summands of~$\polytope{P}$ of dimension at least~$1$ (note that $0$-dimensional summands account for the space of translations).

There are several linearly isomorphic presentations of this cone. The first ones we are aware of are due to McMullen~\cite{McMullen-typeCone} and Meyer~\cite{Meyer}, even though a description was already implicit in previous work of Shephard~\cite{Shephard}. In fact, the type cone can also be reinterpreted as a chamber of regular triangulations of a vector configuration, as introduced in the theory of secondary polytopes~\cite{GelfandKapranovZelevinsky}, see \cite[Sect.~9.5]{DeLoeraRambauSantos} for details. Other formulations can be found, for example, in the appendix of~\cite{PostnikovReinerWilliams}.

The following convenient formulation from~\cite[Lem.~2.1]{ChapotonFominZelevinsky} shows that the deformation cone of simple polytopes is isomorphic to a polyhedral cone, and provides an explicit inequality description, usually called the \defn{wall-crossing inequalities}.

Let $\polytope{P} \subset \R^d$ be a polytope with normal fan $\Fan$ supported on the vector set~$\b{S}$.
Let~$\b{G}$ be the $N \times d$-matrix whose rows are the vectors in~$\b{S}$. 
For any height vector~$\b{h} \in \R^N$, we define the polytope
\(
\polytope{P}_\b{h} \eqdef \set{\b{x} \in \R^d}{\b{G}\b{x} \le \b{h}}.
\)
It is not hard to see that any weak Minkowski summand of~$\polytope{P}$ is of the form $\polytope{P}_\b{h}$ for some $\b{h} \in \R^N$. 

Moreover, for deformations $\polytope{P}_{\b{h}}$ and $\polytope{P}_{\b{h}'}$ of $\polytope{P}$, we have $\polytope{P}_{\b{h}} + \polytope{P}_{\b{h}'} = \polytope{P}_{\b{h}+\b{h}'}$ and $\lambda \polytope{P}_{\b{h}} = \polytope{P}_{\lambda\b{h}}$ for any $\lambda>0$. Hence, the deformation cone, which lies in the space of virtual polyhedra, is linearly isomorphic to the cone
\[
\set{\b{h}\in\R^N}{\polytope{P}_{\b{h}} \in \deformationCone(\polytope{P})}.
\]

\begin{proposition}[\cite{GelfandKapranovZelevinsky,ChapotonFominZelevinsky}]
\label{prop:characterizationPolytopalFan}
Let~$\polytope{P} \subset \R^d$ be a simple polytope with simplicial normal fan~$\Fan$ supported on the rays~$\b{S}$.
Then the deformation cone~$\deformationCone(\polytope{P})$ is the set of polytopes~$\polytope{P}_\b{h}$ for all~$\b{h}$ in the cone of~$\R^\b{S}$ defined by the inequalities
\[
\sum_{\b{s} \in \b{R} \cup \b{R}'} \coefficient[{\b{s}}][\b{R}][\b{R}'] \, \b{h}_{\b{s}} \geq  0
\]
for all adjacent maximal cones~$\R_{\ge0}\b{R}$ and~$\R_{\ge0}\b{R}'$ of~$\Fan$ with~$\b{R} \ssm \{\b{r}\} = \b{R}' \ssm \{\b{r}'\}$, where~$\coefficient[{\b{s}}][\b{R}][\b{R}']$ denote the coefficients in the unique linear dependence\footnote{The linear dependence is unique up to rescaling, and we fix this arbitrary positive rescaling for convenience in the exposition.}
\[
\sum_{\b{s} \in \b{R} \cup \b{R}'} \coefficient[{\b{s}}][\b{R}][\b{R}'] \, \b{s} = \b{0}
\]
among the rays of~$\b{R} \cup \b{R}'$ such that~$\coefficient[{\b{r}}][\b{R}][\b{R}'] + \coefficient[{\b{r}'}][\b{R}][\b{R}'] = 2$. 
\end{proposition}

This characterization can be extended to general (not necessarily simple) polytopes. One straightforward way to do so is via a simplicial refinement of the normal fan. If such a simplicial refinement contains additional rays, then the type cone will be embedded in a higher dimensional space, but projecting out these additional coordinates gives a linear isomorphism with the standard presentation. See \mbox{\cite[Prop.~3]{PilaudSantos-quotientopes}} and \mbox{\cite[Prop.~1.7]{PadrolPaluPilaudPlamondon}}.

\begin{proposition}
\label{prop:non-simplicial}
Let $\polytope{P} \subset \R^d$ be a polytope whose normal fan $\Fan$ is refined by the simplicial fan~$\Fan[G]$ supported on the rays~$\b{S}$.
Then the deformation cone~$\deformationCone(\polytope{P})$ is the set of polytopes~$\polytope{P}_\b{h}$ for all~$\b{h}$ in the cone of~$\R^\b{S}$ defined by
\begin{itemize}
\item the equalities~$\sum_{\ray[s] \in \rays \cup \rays'} \coefficient[{\ray[s]}][\rays][\rays'] \, \b{h}_{\ray[s]} = 0$ for any adjacent maximal cones~$\R_{\ge0}\rays$ and~$\R_{\ge0}\rays'$ of~$\Fan[G]$ belonging to \textbf{the same} maximal cone of $\Fan$,
\item the inequalities~$\sum_{\ray[s] \in \rays \cup \rays'} \coefficient[{\ray[s]}][\rays][\rays'] \, \b{h}_{\ray[s]} \geq  0$ for any adjacent maximal cones~$\R_{\ge0}\rays$ and~$\R_{\ge0}\rays'$ of~$\Fan[G]$ belonging to \textbf{distinct} maximal cones of $\Fan$,
\end{itemize}
where
\(
\sum_{\b{s} \in \b{R} \cup \b{R}'} \coefficient[{\b{s}}][\b{R}][\b{R}'] \, \b{s} = \b{0}
\)
is the unique linear dependence with~${\coefficient[{\b{r}}][\b{R}][\b{R}'] + \coefficient[{\b{r}'}][\b{R}][\b{R}'] = 2}$ among the rays of two adjacent maximal cones~$\R_{\ge0}\b{R}$ and~$\R_{\ge0}\b{R}'$ of~$\Fan[G]$ with~$\b{R} \ssm \{\b{r}\} = \b{R}' \ssm \{\b{r}'\}$.
\end{proposition}


\subsection{Graphical zonotopes}
\label{subsec:graphicalZonotopes}

Let $G \eqdef (V,E)$ be a graph with vertex set $V$ and edge set~$E$. 
The \defn{graphical arrangement}~$\gArr$ is the arrangement of the hyperplanes $\set{\b{x} \in \R^V}{\b{x}_u = \b{x}_v}$ for all edges ${\{u,v\} \in E}$. 
It induces the \defn{graphical fan} $\gFan$ whose cones are all the possible intersections of one of the sets~$\set{\b{x} \in \R^V}{\b{x}_u = \b{x}_v}$, $\set{\b{x} \in \R^V}{\b{x}_u \ge \b{x}_v}$, or $\set{\b{x} \in \R^V}{\b{x}_u \le \b{x}_v}$ for each edge $\{u,v\}\in E$. 
The lineality of $\gFan$ is the subspace~$\cc$ of~$\R^V$ spanned by the characteristic vectors of the connected components of~$G$.

The \defn{graphical zonotope}~$\gZono$ is the Minkowski sum of the line segments $[\b{e}_u,\b{e}_v]$ in~$\R^V$ for all edges ${\{u,v\}\in E}$.
Here, $(\b{e}_v)_{v \in V}$ denotes the canonical basis of~$\R^V$. Note that~$\gZono$ lies in a subspace orthogonal to~$\cc$. The graphical fan~$\gFan$ is the normal fan of the graphical zonotope~$\gZono$.

The following result is well-known. For example, it can be easily deduced from \cite[Prop.~2.5]{Stanley2007} or \cite{OrientedMatroids} (for the latter, see that the graphical matroid from Sec.~1.1 is realized by the graphical arrangement, and use the description of the cells of the arrangement in terms of covectors from Sec.~1.2(c)).

An \defn{ordered partition} $(\mu,\omega)$ of~$G$ consists of a partition $\mu$ of $V$ where each part induces a connected subgraph of~$G$, together with an acyclic orientation $\omega$ of the quotient graph $G/\mu$. We say that $(\mu,\omega)$ refines $(\mu',\omega')$ if each part of $\mu$ is contained in a part of $\mu'$ and the orientations are compatible; that is, for all $u,v\in V$ if there is a directed path in $\omega$ between the parts of $\mu$ respectively containing $u$ and $v$, then there is a directed path in $\omega'$ between the parts of $\mu'$ respectively containing $u$ and $v$. 

\begin{proposition} 
\label{prop:facesGraphicalZonotope}
The face lattice of~$\gFan$ is antiisomorphic to the lattice of ordered partitions of~$G$ ordered by refinement.
Explicitly, the antiisomorphism is given by the map that associates the ordered partition $(\mu,\omega)$ to the cone $\opcone$ defined by the inequalities $\b{x}_u \leq \b{x}_v$ for all $u,v\in V$ such that there is a directed path in $\omega$ from the part containing~$u$ to the part containing~$v$ (in particular, $\b{x}_u = \b{x}_v$ if $u,v$ are in the same part of~$\mu$).
\end{proposition}

Some easy consequences of \cref{prop:facesGraphicalZonotope} are:
\begin{itemize}
\item The maximal cones of~$\gFan$ are in bijection with the acyclic orientations of~$G$. We denote by~$\orcone$ the maximal cone of~$\gFan$ associated to the acyclic orientation~$\omega$.
\item The minimal cones of $\gFan$, that is the rays of $\gFan/\cc$, are in bijection with the \defn{biconnected subsets} of~$G$, \ie non-empty subsets of~$V$ spanning connected subgraphs whose complements in their connected components are also non-empty and connected. 
\item The rays of $\gFan/\cc$ that belong to the maximal cone associated to an acyclic orientation are the biconnected subsets which form an upper set of the acyclic orientation (hence, they are in bijection with the minimal directed cuts of the acyclic orientation).
\item Similarly, the rays of $\gFan/\cc$ that belong to the cone associated to an ordered partition $(\mu,\omega)$ are the biconnected sets that contracted by $\mu$ give rise to an upper set of~$\omega$.  
 \end{itemize}

Note that the natural embedding of a graphical fan~$\gFan$ is not essential, as it has a lineality given by its connected components. This is why we cannot directly talk about the rays of the fan in the enumeration above. The usual solution to avoid this is to consider the quotient by the subspace~$\cc$. However, this subspace depends on the graph, and with such a quotient we would lose the capacity of uniformly treating all the graphs with a fixed vertex set. 
We will instead work with the natural non-essential embedding, together with a collection of vectors supporting simultaneously all graphical fans.

\begin{example}\label{ex:complete}
When $G$ is the complete graph~$K_n$, the graphical zonotope is the \defn{permutahedron}. The graphical fan is the \defn{braid fan}~$\braid$, induced by the \defn{braid arrangement} consisting of the hyperplanes $\set{\b{x} \in \R^n}{\b{x}_i = \b{x}_j}$ for all $1\leq i<j\leq n$. Its lineality is spanned by the all-ones vector~$\ones \eqdef (1, \dots, 1)$. 
Since all the subsets of~$[n]\eqdef \{1,2,\dots,n\}$ are biconnected in~$K_n$, the face lattice of~$\braid$ is isomorphic to the lattice of ordered partitions of~$[n]$. The rays of~$\braid/\ones$ correspond to proper subsets of~$[n]$, and its maximal cells are in bijection with permutations of~$[n]$. Each maximal cell is the positive hull of the $n-1$ rays corresponding to the proper upper sets of the order given by the permutation. In particular, $\braid/\ones$ is a simplicial fan. 
\end{example}


\section{Graphical deformation cones}
\label{sec:graphicalTypeCones}

Our main result is an irredundant facet description of the deformation cone of~$\gZono$ for every graph~$G \eqdef (V,E)$. Our starting point is \cref{prop:redundant}, which gives a (possibly redundant) description derived from \cref{prop:non-simplicial}. 
It is strongly based on the fact that the braid fan simultaneously refines all the graphical fans. 
Note however that the braid fan is not simplicial (due to its lineality).
The classical approach to overcome this issue is to quotient the braid fan by its lineality space.
However, we prefer to triangulate the braid fan, since it simplifies the presentation of the proof.


\subsection{A first polyhedral description}
\label{subsec:firstDescription}

Associate to each subset $S \subseteq V$ the vector
\[
{\ivector{S} \eqdef  \sum_{v \in S} \b{e}_v  - \sum_{v \notin  S} \b{e}_v}.
\]
This is essentially the characteristic vector of $S$, but it has the advantage that $\ivector{V} = \ones[V]$ and $\ivector{\varnothing} = -\ones[V]$ positively span the line $\ones[V]\mathbb{R}$, which is the lineality~$\cc[K_V]$ of the braid fan.

\begin{lemma}
For any ordered partition~$(\mu,\omega)$ of a graph~$G \eqdef (V,E)$, we have
\[
\opcone=\cone\set{\ivector{S}}{S\subseteq V \text{ upper set of } \omega}.
\]
\end{lemma}

Here, we mean that~$S$ is an upper set of~$\omega$ when contracted by~$\mu$.
Note that $\varnothing$ and $V$ are always upper sets, which is consistent with the fact that the lineality of~$\gFan$ always contains the line spanned by~$\ones[V]$.

We will work with a refined version~$\sbraid[V]$ of the braid fan whose maximal cells are 
\[
\opcone[\sigma]^\varnothing \eqdef \cone\set{\ivector{S}}{S\subsetneq V\text{ upper set of }\sigma}
\quad\text{and}\quad
\opcone[\sigma]^V \eqdef \cone\set{\ivector{S}}{\varnothing\neq S\subseteq V\text{ upper set of }\sigma}
\]
for every acyclic orientation of $K_{V}$, which we identify with a permutation~$\sigma$ of~$V$. An example is depicted in \cref{fig:deformedbraidfan}.
The following two immediate statements are left to the reader.

\begin{lemma}\label{lem:braidcircuits}For any finite set~$V$:
\begin{enumerate}[(i)]
\item The fan~$\sbraid[V]$ is an essential complete simplicial fan in~$\R^V$ supported on the $2^{|V|}$ vectors~$\ivector{S}$ for $S\subseteq V$. 
 
\item For any permutation~$\sigma$, the maximal cones $\opcone[\sigma]^\varnothing$ and $\opcone[\sigma]^V$ are adjacent, and the unique linear relation supported on the rays of $\opcone[\sigma]^\varnothing\cup\opcone[\sigma]^V$ is
\(
\ivector{\varnothing} + \ivector{V} = \b{0}.
\)

\item The other pairs of adjacent maximal cells are of the form $\opcone[\sigma]^X$ and $\opcone[\sigma']^X$, where $X\in\{\varnothing,V\}$ and $\sigma=PuvS$ and $\sigma'=PvuS$ are permutations that differ in the inversion of two consecutive elements.
The two rays that are not shared by~$\opcone[\sigma]^X$ and~$\opcone[\sigma']^X$ are $\ivector{S\cup \{u\}}$ and $\ivector{S\cup \{v\}}$, and the unique linear relation supported on the rays of~$\opcone[\sigma]^X\cup\opcone[\sigma']^X$ is given by 
\[
\ivector{S \cup \{u\}} + \ivector{S \cup \{v\}} = \ivector{S} + \ivector{S \cup \{u, v\}}.
\]
\end{enumerate}
\end{lemma}

\begin{figure}
	\includegraphics[width=.3\linewidth]{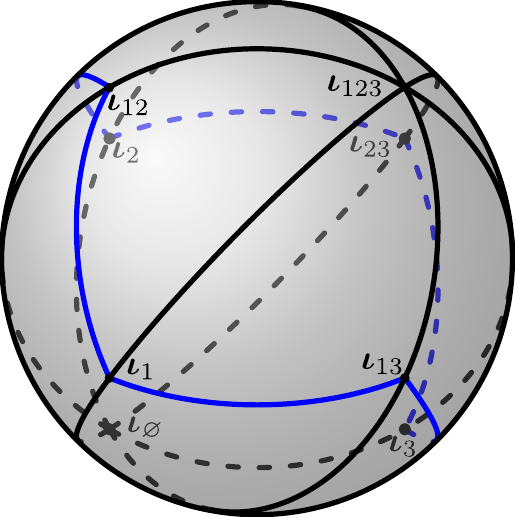}
	\caption{The fan~$\sbraid[123]$ intersected with the unit sphere. (For brevity, here and in the labels we write $123$ to denote the set $\{1,2,3\}$, and so on.) The braid fan~$\braid[123]$ is the Cartesian product of a regular hexagonal fan with a line. To obtain~$\sbraid[123]$, each maximal cell is divided into two simplicial cells, one containing~$\ivector{\varnothing}$ and one containing~$\ivector{123}$. }
	\label{fig:deformedbraidfan}
\end{figure}

\begin{lemma}\label{lem:braidrefining}
For any graph~$G \eqdef (V,E)$: 
\begin{enumerate}[(i)]
\item The fan $\sbraid[V]$ is a simplicial refinement of the graphical~fan~$\gFan$.
 
\item For an acyclic orientation~$\omega$ of~$G$ and $S\subseteq V$, we have $\ivector{S}\in \orcone$ if and only if $S$ is an upper set of $\omega$.

\item For an acyclic orientation $\sigma$ of $K_V$ and $X\in\{\varnothing,V\}$ we have $\orcone[\sigma]^X\subseteq\orcone$ if and only if~$\sigma$ is a linear extension of~$\omega$.
\end{enumerate}
\end{lemma}

We are now ready to describe the deformation cone of the graphical zonotope~$\gZono$.
For any $\b{h} \in \R^{2^V}$, let~$\dZono$ be the deformed permutahedron given by
\[
\dZono \eqdef \Bigset{\b{x} \in \R^V}{\sum_{v \in S} \b{x}_v - \sum_{v \notin S} \b{x}_v \le \b{h}_S \text{ for all } S \subseteq V}.
\]

\begin{proposition}\label{prop:redundant}
For any graph $G \eqdef (V,E)$, the deformation cone~$\deformationCone(\gZono)$ of the graphical zonotope~$\gZono$ is the set of polytopes~$\dZono$ for all~$\b{h}$ in the cone of~$\R^{2^V}$ defined by the following (possibly redundant) 
description:
\begin{itemize}
\item $\b{h}_{\varnothing}=-\b{h}_{V}$,
\item $\b{h}_{S \cup \{u\}} + \b{h}_{S \cup \{v\}} = \b{h}_{S} + \b{h}_{S \cup \{u, v\}}$ for each~$\{u,v\}\in \binom{V}{2}\ssm E$ and~$S\subseteq V\ssm \{u,v\}$, and
\item $\b{h}_{S \cup \{u\}} + \b{h}_{S \cup \{v\}} \geq \b{h}_{S} + \b{h}_{S \cup \{u, v\}}$ for each~$\{u,v\}\in E$ and~$S\subseteq V\ssm \{u,v\}$.
\end{itemize}
\end{proposition}
 
\begin{proof} Observe first that, as stated in \cref{lem:braidrefining}, $\sbraid[V]$ provides a simplicial refinement of~$\gFan$. 
Following \cref{prop:non-simplicial}, we need to consider all pairs of adjacent maximal cones of~$\sbraid[V]$, and to study which ones lie in the same cone of~$\gFan$. 

Adjacent maximal cones of~$\sbraid[V]$ are described in \cref{lem:braidcircuits}, and the containement relations of the cones of~$\sbraid[V]$ in the cones of~$\gFan$ are described in \cref{lem:braidrefining}.

For any $\sigma$, the cones~$\opcone[\sigma]^\varnothing$ and~$\opcone[\sigma]^V$ belong to the same cell of~$\gFan$. Hence, by \cref{prop:non-simplicial}, the following equation holds in the deformation cone:
\[
\b{h}_{\varnothing}=-\b{h}_{V}.
\]

The remaining pairs of adjacent maximal cones of~$\sbraid[V]$ correspond to pairs of acyclic orientations of $K_V$ differing in a single edge; or equivalently, to  pairs of permutations of~$V$ of the form $\sigma=PuvS$ and $\sigma'=PvuS$.
The unique linear relation supported on the rays of $\orcone[\sigma]^X\cup \orcone[\sigma']^X$ for $X\in\{\varnothing,V\}$ is then 
 \[\ivector{S \cup \{u\}} + \ivector{S \cup \{v\}} = \ivector{S} + \ivector{S \cup \{u, v\}}.\]

We consider first the case when $\{u,v\}\notin E$. Observe that both $\sigma$ and $\sigma'$ induce the same 
acyclic orientation of $G$, which we call~$\omega$. We have then $\orcone[\sigma]^X\cup \orcone[\sigma']^X\subseteq \orcone[\omega]$ by \cref{lem:braidrefining}. 
Therefore, by \cref{prop:non-simplicial} and \cref{lem:braidcircuits}, we have 
\[\b{h}_{S \cup \{u\}} + \b{h}_{S \cup \{v\}} = \b{h}_{S} + \b{h}_{S \cup \{u, v\}}\]
for any $\b{h}$ in~$\deformationCone(\gZono)$.
Note that, for any $\{u,v\}\notin E$ and $S\subset V\ssm\{u,v\}$, we can construct such permutations $\sigma$ and $\sigma'$.
This gives the claimed description of the linear span of~$\deformationCone(\gZono)$.

In contrast, if $\{u,v\}\in E$, then $\sigma$ and $\sigma'$ induce different orientations of~$G$, and hence they belong to different adjacent cones of~$\gFan$ by \cref{lem:braidrefining}. Therefore, by \cref{prop:non-simplicial} and \cref{lem:braidcircuits}, we have 
\[\b{h}_{S \cup \{u\}} + \b{h}_{S \cup \{v\}} \geq  \b{h}_{S} + \b{h}_{S \cup \{u, v\}}\]
for any $\b{h}$ in~$\deformationCone(\gZono)$. 
As before, for any $\{u,v\}\in E$ and $S\subset V\ssm\{u,v\}$, we can construct such permutations $\sigma$ and $\sigma'$.
This gives the claimed inequalities describing~$\deformationCone(\gZono)$.  
\end{proof}

 
\subsection{The linear span of graphical deformation cones}
\label{subsec:linearSpan}

The description of the deformation cone of \cref{prop:redundant} is highly redundant, both in the equations describing its linear span and in the inequalities describing its facets. 
We will give a non-redundant description in \cref{thm:main}. The first step will be to give linearly independent equations describing the linear span. As a by-product, we will obtain the dimension and a linear basis of the deformation cone~$\deformationCone(\gZono)$.

For a polytope~$\polytope{P} \subset \R^d$, we define the space $\VD(\polytope{P}) \subset \virtualPolytopes[d]$ of \defn{virtual deformations} of~$\polytope{P}$ as the vector subspace of virtual polytopes generated by the deformations of~$\polytope{P}$. Equivalently, $\VD(\polytope{P})$ is the linear span of the deformation cone~$\deformationCone(\polytope{P})$.
Every virtual polytope in~$\VD(\polytope{P})$ is of the form $\polytope{P}_\b{h} - \polytope{P}_\b{h'}$ for deformations $\polytope{P}_\b{h}, \polytope{P}_\b{h'}\in \deformationCone(P)$. Note that the vector $\b{h}-\b{h'}$ uniquely describes the equivalence class of this virtual polytope, and we will use the notation $\polytope{P}_{\b{h}-\b{h'}}$ to denote it.

Denote by~$\simplex_U \eqdef \conv\set{\b{e}_u}{u \in U}$ the face of the standard simplex~$\simplex_V$ corresponding to a subset~$U \subseteq V$.
These polytopes are particularly important deformed permutahedra as they form a linear basis of the deformation space of the permutahedron \cite[Prop.~2.4]{ArdilaBenedettiDoker}. Namely, any (virtual) deformed permutahedron can be uniquely written as a signed Minkowski sum of dilates of~$\simplex_I$.
Our first result states that this linear basis is adapted to graphical zonotopes. 

\begin{theorem}
\label{thm:basisAndEquationsGraphicalTypeCones}
For any graph $G \eqdef (V,E)$:
\begin{enumerate}[(i)]
 \item The dimension of~$\VD(\gZono)$ is the number of non-empty induced cliques in~$G$ (the vertices of~$G$ count for the dimension as they correspond to the lineality space).
 \item\label{item:cliquesbasis} The faces~$\simplex_K$ of the standard simplex~$\simplex_V$ corresponding to the non-empty induced cliques~$K$ of~$G$ form a linear basis of~$\VD(\gZono)$.
 \item $\VD(\gZono)$ is the set of virtual polytopes~$\dZono$ for all~$\b{h} \in \R^{2^V}$ fulfilling the following linearly independent equations:
	\begin{itemize}
	\item $\b{h}_{\varnothing}=-\b{h}_{V}$ and
	\item $\b{h}_{S \ssm \{u\}} + \b{h}_{S \ssm \{v\}} = \b{h}_S + \b{h}_{S \ssm \{u, v\}}$ for each $S\subseteq V$ with $|S|\geq 2$ not inducing a clique of~$G$ and any
$\{u,v\}\in \binom{S}{2}\ssm E$ (here, we only choose one missing edge for each subset~$S$, for example, the lexicographically smallest).
	\end{itemize}
\end{enumerate}
\end{theorem}

\begin{proof}
Observe first that the faces~$\simplex_K$ of the standard simplex~$\simplex_V$ corresponding to the induced cliques~$K$ of~$G$ are all in the  deformation cone~$\deformationCone(\gZono)$.
Indeed, faces of the standard simplex~$\Delta_K$ belong to the deformation cone of the complete graph~$K_K$ by \cite[Prop.~6.3]{Postnikov}.
The graphical zonotope $\gZono[G']$ is a Minkowski summand of~$\gZono$ for any subgraph~$G'$ of~$G$, and hence summands of $\gZono[G']$ are also summands of $\gZono[G]$. 

Moreover, all faces~$\simplex_I$ for~${\varnothing \ne I \subsetneq V}$ are Minkowski independent by~\cite[Prop.~2.4]{ArdilaBenedettiDoker}.
This shows that the dimension of $\VD(\gZono)$ is at least the number of non-empty induced cliques of~$G$.

Let $(\b{f}_X)_{X\subseteq V}$ be the canonical basis of~$\big(\mathbb{R}^{2^V}\big)^*$.
The vectors 
\[
\b{o}^S\eqdef \b{f}_S -\b{f}_{S \ssm \{u\}} - \b{f}_{S \ssm \{v\}} + \b{f}_{S \ssm \{u, v\}},
\]
for all subsets $\varnothing \neq S \subseteq V$ not inducing a clique of~$G$ and one selected missing edge $\{u, v\}$ for each~$S$, are clearly linearly independent. 
Indeed, if the $\b{f}_X$ are ordered according to any linear extension of the inclusion order on the indices~$X$, and the $\b{o}^S$ are ordered analogously in terms of the indices~$S$, then the equations are already in echelon form, as $\b{f}_S$ is the greatest non-zero coordinate of~$\b{o}^S$.
Finally, the vector $\b{v}\in 2^V$ with $\b{v}_X=|X|$ for $X\in 2^V$ is orthogonal to any $\b{o}^S$ with $|S| \geq 2$ but not to $\b{o}^\varnothing\eqdef \b{f}_{\varnothing}+\b{f}_{V}$, showing that the latter is linearly independent to the former. This proves that the dimension of $\VD(\gZono)$ is at most the number of non-empty induced cliques of~$G$.

We conclude that $\bigset{\simplex_K}{\varnothing\neq K\subseteq V\text{ inducing a clique of }G}$ is a linear basis of the deformation cone, and that $\bigset{\b{o}^S}{S = \varnothing \text{ or } S \subseteq V \text{ not inducing a clique of } G}$ is a basis of its orthogonal complement (we slightly abuse notation here as  $\b{o}^S$ was defined in~$\smash{\big(\mathbb{R}^{2^V}\big)^*}$ instead of in~$(\virtualPolytopes)^*$, but note that each $\b{f}_X$ can be considered as a linear functional in~$(\virtualPolytopes)^*$ if seen as a support function).
\end{proof}

Note that the dimension of the deformation space of graphical zonotopes has been independently computed by Raman Sanyal and Josephine Yu (personnal communication), who computed the space of Minkowski $1$-weights of graphical zonotopes in the sense of McMullen~\cite{McMullen1996}. Their proof also uses the basis from \cref{thm:basisAndEquationsGraphicalTypeCones}~\ref{item:cliquesbasis}, but with an alternative argument to show that they are a generating family.


\subsection{The facets of graphical deformation cones}
\label{subsec:facets}

To conclude, it remains to compute the facets of the deformation cones, \ie a non-redundant inequality description.

We define the \defn{neighborhood} of a vertex~$v$ of a graph $G \eqdef (V,E)$ as $N(v)\eqdef \set{u\in V}{\{u,v\}\in E}$.

\begin{theorem}\label{thm:main}
For any graph $G \eqdef (V,E)$, the deformation cone~$\deformationCone(\gZono)$ of the graphical zonotope~$\gZono$ is the set of polytopes~$\dZono$ for all~$\b{h}$ in the cone of~$\R^{2^V}$ defined by the following irredundant facet description:
\begin{itemize}
\item $\b{h}_{\varnothing}=-\b{h}_{V}$,
\item $\b{h}_{S \ssm \{u\}} + \b{h}_{S \ssm \{v\}} = \b{h}_S + \b{h}_{S \ssm \{u, v\}}$ for each~$\varnothing \neq S \subseteq V$ and any~$\{u,v\} \in \binom{S}{2} \ssm E$,
\item $\b{h}_{S \cup \{u\}} + \b{h}_{S \cup \{v\}} \geq  \b{h}_{S} + \b{h}_{S \cup \{u, v\}}$ for each~$\{u,v\}\in E$ and~$S \subseteq N(v) \cap N(v)$.
\end{itemize}
\end{theorem}

Note that this description is given as a face of the submodular cone, embedded into $\R^{2^V}$. One gets easily an intrinsic presentation by restricting to the space spanned by the biconnected subsets of~$V$. However, that presentation loses its symmetry, and the explicit equations depend on the biconnected sets of~$G$.

\begin{proof}[Proof of \cref{thm:main}]
We know by \cref{prop:redundant} that~$\deformationCone(\gZono)$ is the intersection of the cone  
\begin{equation}\label{eq:inequality}
\b{h}_{S \cup \{u\}} + \b{h}_{S \cup \{v\}} \geq  \b{h}_{S} + \b{h}_{S \cup \{u, v\}}
\end{equation}
for $\{u,v\}\in E \text{ and } S\subseteq  V\ssm\{u,v\}$ with the linear space given by the equations $\b{h}_\varnothing=-\b{h}_V$ and
\begin{equation}\label{eq:equality}
\b{h}_{S \cup \{u\}} + \b{h}_{S \cup \{v\}} = \b{h}_{S} + \b{h}_{S \cup \{u, v\}}
\end{equation}
for $\{u,v\}\in \binom{V}{2}\ssm E \text{ and } S\subseteq  V\ssm\{u,v\}$.

We have already determined the equations describing the linear span in \cref{thm:basisAndEquationsGraphicalTypeCones}, so it only remains to provide non-redundant inequalities describing the deformation cone.
 
We will prove first that the inequalities from \eqref{eq:inequality} indexed by $\{u,v\}\in E \text{ and } S\subseteq  N(v)\cap N(v)$ suffice to describe~$\deformationCone(\gZono)$. 
To this end, consider an inequality from \eqref{eq:inequality} for which $S\nsubseteq  N(v)\cap N(v)$. Without loss of generality, assume that there is some $x\in S$ such that $\{x,v\}\notin E$. We will show that this inequality is induced (in the sense that the halfspaces they define coincide on the linear span of~$\deformationCone(\gZono)$) by the inequality 
\begin{equation}\label{eq:reducedinequality}
 \b{h}_{S' \cup \{u\}} + \b{h}_{S' \cup \{v\}} \geq \b{h}_{S'} + \b{h}_{S' \cup \{u, v\}}
\end{equation}
where $S'=S\ssm \{x\}$. Our claim will then follow from this by induction on the elements of $S\ssm (N(v)\cap N(v))$.

Indeed, if $\{x,v\}\notin E$, we know by \eqref{eq:equality} that the following two equations hold in the linear span of~$\deformationCone(\gZono)$ by considering the non-edge $\{x,v\}$ with the subsets $S'$ and $S'\cup\{u\}$, respectively: 
\begin{align}
\b{h}_{S\cup \{u\}} + \b{h}_{S'\cup \{u,v\}} &= \b{h}_{S'\cup \{u\}} + \b{h}_{S\cup \{u, v\}}\label{eq:auxeq1},\\
\b{h}_{S} + \b{h}_{S' \cup \{v\}} &= \b{h}_{S'} + \b{h}_{S \cup \{ v\}}\label{eq:auxeq2},
\end{align}
where we used that $(S'\cup \{u\}) \cup \{x\}=S\cup \{u\}$ and $(S'\cup \{u\}) \cup \{x, v\}=S\cup \{u, v\}$ in the first equation, and that $S' \cup \{x\}=S$ and $S' \cup \{x, v\}=S \cup \{ v\}$ in the second equation.
To conclude, note that \eqref{eq:inequality} is precisely the linear combination $\eqref{eq:reducedinequality} + \eqref{eq:auxeq1} - \eqref{eq:auxeq2}$.

We know therefore that the descriptions in \cref{prop:redundant} and \cref{thm:main} give rise to the same cone. It remains to show that the latter is irredundant. That is, that each of the inequalities gives rise to a unique facet of~$\deformationCone(\gZono)$.

Let $(\b{f}_X)_{X\subseteq V}$ be the canonical basis of $\big(\mathbb{R}^{2^V}\big)^*$.
For $u,v\in V$ and $S\subseteq V\ssm\{u,v\}$, let 
\[
\b{n}(u,v,S) \eqdef \b{f}_{S \cup \{u\}} + \b{f}_{S \cup \{v\}} - \b{f}_{S} - \b{f}_{S \cup \{u, v\}}.
\]
Note that, if $\{u,v\}\notin E$, then $\b{n}(u,v,S)$ is orthogonal to~$\deformationCone(\gZono)$, whereas if $\{u,v\}\in E$, then $\b{n}(u,v,S)$ is an inner normal vector to~$\deformationCone(\gZono)$.

Fix $\{u,v\}\in E$ and $S\subseteq  N(v)\cap N(v)$. To prove that the halfspace with normal~$\b{n}(u,v,S)$ is not redundant, we will exhibit a vector $\b{w}\in \mathbb{R}^{2^V}$ in the linear span of~$\deformationCone(\gZono)$ that belongs to the interior of all the halfspaces describing~$\deformationCone(\gZono)$ except for this one. 
That is, we will construct a vector $\b{w}\in \mathbb{R}^{2^V}$ respecting the system:
\begin{equation}\label{eq:system}
\begin{cases}
     \dotprod{ \b{w} }{\b{n}(u,v,S)} \leq  0,\\
     \dotprod{ \b{w} }{\b{n}(u,v,X)} >  0  &\text{ for }  S\ne X \subseteq N(u)\cap N(v),\\
     \dotprod{ \b{w} }{\b{n}(a,b,X)} >  0  &\text{ for }  \{a,b\}\in E\ssm \{u,v\}\text{ and } X\subseteq N(a)\cap N(b),\text{ and } \\
     \dotprod{ \b{w} }{\b{n}(a,b,X)} =  0  &\text{ for } \{a,b\}\in \binom{V}{2}\ssm E \text{ and } X\subseteq V\ssm\{a,b\}.
\end{cases}
\end{equation}

Denote by $T \eqdef N(u)\cap N(v)\ssm S$. We will construct $\b{w}$ as the sum $\b{w}\eqdef \b{t}^S - \b{t}^T+ \b{c}$ for some vectors $\b{t}^S$, $\b{t}^T$, and $\b{c}\in \mathbb{R}^{2^V}$ defined below whose scalar products with $\b{n}(a,b,X)$ for $\{a,b\}\in \binom{V}{2}$ and $X\subseteq V\ssm \{a,b\}$ fulfill:
\[
\renewcommand{\arraystretch}{1.3}
\begin{array}{l|ccc}
& \dotprod{\b{t}^S}{\b{n}(a,b,X)}& \dotprod{-\b{t}^T}{\b{n}(a,b,X)} & \dotprod{\b{c}}{\b{n}(a,b,X)} \\
\hline
\text{if } \{a,b\}=\{u,v\}\text{ and }X=S& -|S| & 0  & |S| \\
\text{if } \{a,b\}=\{u,v\}\text{ and }S\ne X \subseteq N(u)\cap N(v)& -|S\cap X| & |T\cap X| & |S|\\
\text{if } \{a,b\}\in E\ssm\{u,v\}\text{ and }X\subseteq N(a)\cap N(b)& \geq -1 & \geq 0 & 2\\
\text{if } \{a,b\}\notin E& 0 & 0  & 0 
\end{array}
\renewcommand{\arraystretch}{1}
\]
It immediately follows from this table that the vector $\b{w}$ will fulfill the desired properties from~\eqref{eq:system}. For the second one, note that if $S\ne X \subseteq S\sqcup T$, then either $|S\cap X|<|S|$ or $|T\cap X|>0$.

To define these vectors,  first, for $\{x,y,z\}\in \binom{V}{3}$, let $\b{t}^{xyz}\in \mathbb{R}^{2^V}$ be the vector such that $\b{t}^{xyz}_X = 1$ if $\{x,y,z\} \subseteq X$ and $\b{t}^{xyz}_X = 0$ otherwise. Note that, for any $a,b\in \binom{V}{2}$ and $X\subseteq V\ssm\{a,b\}$, we have
\begin{equation}\label{eq:txyz}
\dotprod{\b{t}^{xyz}}{\b{n}(a,b,X)} = 
\begin{cases}
-1& \text{ if }\{x,y,z\}=\{a,b,t\}\text{ for some }t\in X, \text{ and }\\
0 & \text{ otherwise }.
\end{cases}
\end{equation}
We define 
\[
\b{t}^S \eqdef \sum_{s\in S} \b{t}^{uvs}
\qquad\text{ and }\qquad \b{t}^T \eqdef \sum_{t\in T} \b{t}^{uvt}.
\]
It is straightforward to derive the identities in the table from \eqref{eq:txyz}. For the inequalities, notice that 
if $\dotprod{\b{t}^{uvx}}{\b{n}(a,b,X)}=-1$ but $\{a,b\}\neq \{u,v\}$, then either $\{a,b\}=\{u,x\}$ or $\{a,b\}=\{v,x\}$, and in both cases  
 $\dotprod{\b{t}^{uvy}}{\b{n}(a,b,X)}=0$ for any $y\neq x$.

Now, for $\{x,y\}\in \binom{V}{2}$, let $\b{c}^{xy}\in \mathbb{R}^{2^V}$ be the vector such that $\b{c}^{xy}_X = 1$ if $| \{x,y\} \cap X| =1$ (that is, if $\{x,y\}$ belongs to the cut defined by~$X$), and $\b{c}^{xy}_X = 0$ otherwise. Note that, for any $a,b\in \binom{V}{2}$ and $X\subseteq V\ssm\{a,b\}$, we have 
\begin{equation}\label{eq:cxy}
\dotprod{\b{c}^{xy}}{\b{n}(a,b,X)} =
\begin{cases}
2&\text{ if }\{a,b\}= \{x,y\},\text{ and }\\
0&\text{ otherwise.}
\end{cases}
\end{equation}
We set 
\[
\b{c}\eqdef \frac{|S|}{2} \b{c}^{uv} + \sum_{ \{a,b\}\in E\ssm \{u,v\}} \b{c}^{ab}.
\]
The identities in the table are straightforward to derive from~\eqref{eq:cxy}.
\end{proof}

\begin{corollary}
\label{coro:numerologyGraphicalTypeCone}
For any graph~$G \eqdef (V,E)$, the dimension of~$\deformationCone(\gZono)$ is the number of induced cliques in~$G$, the dimension of the lineality space of~$\deformationCone(\gZono)$ is~$|V|$, and the number of facets of~$\deformationCone(\gZono)$ is the number of triplets~$(u,v,S)$ with~$\{u,v\} \in E$ and~$S \subseteq N(u) \cap N(v)$.
\end{corollary}

\begin{example}
For the complete graph~$K_V$, the graphical zonotope $\gZono[K_V]$ is a permutahedron and the deformation cone $\deformationCone(\gZono[K_V])$ is the submodular cone given by the irredundant inequalities~${\b{h}_{S \cup \{u\}} + \b{h}_{S \cup \{v\}} \geq  \b{h}_{S} + \b{h}_{S \cup \{u, v\}}}$ for each~$\{u,v\} \subseteq V$ and~$S \subseteq V \ssm \{u,v\}$. 
(The usual presentation imposes $\b{h}_\varnothing = 0$, but both presentations are clearly equivalent up to translation).
It has dimension~$2^{|V|}-1$ and $\binom{|V|}{2} 2^{|V|-2}$ facets.
The lineality is $|V|$-dimensional, given by the space of translations in~$\R^{|V|}$.

For instance, for the triangle~$K_3$, the graphical zonotope~$\gZono[K_3]$ is the regular hexagon depicted in the bottom left of \cref{fig:3CycleDeformationCone}, which arises as the Minkowski sum of $3$ coplanar vectors in $\R^{3}$.
Its deformation cone $\deformationCone(\gZono[K_3])$ lives in the \mbox{$8$-dimensional} space~$\smash{\R^{2^{[3]}}}$, has dimension~$7$, a lineality space of dimension~$3$, and $6$ facets.
It admits as irredundant description the equation $\b{h}_{\varnothing} = -\b{h}_{123}$ and the following $6$ inequalities:
\begin{align*}
    \b{h}_1 + \b{h}_2 & \ge \b{h}_{\varnothing} + \b{h}_{12} &
    \b{h}_1 + \b{h}_3 & \ge \b{h}_{\varnothing} + \b{h}_{13} &
    \b{h}_2 + \b{h}_3 & \ge \b{h}_{\varnothing} + \b{h}_{23}
    \\
    \b{h}_{12} + \b{h}_{13} & \ge \b{h}_1 + \b{h}_{123} &
    \b{h}_{12} + \b{h}_{23} & \ge \b{h}_2 + \b{h}_{123} &
    \b{h}_{13} + \b{h}_{23} & \ge \b{h}_3 + \b{h}_{123}.
\end{align*}
After quotienting the lineality and intersecting with an affine hyperplane, we get the bipyramid illustrated on \cref{fig:3CycleDeformationCone}.
Note that the four rays of $\deformationCone(\gZono[K_3])$ (\ie vertices of the bipyramid) of the form $\simplex_K$ for an induced clique~$K$ of $K_3$ provide a linear basis of $\deformationCone(\gZono[K_3])$ (\ie an affine basis of the bipyramid).
Nevertheless, the last ray can not be written as a positive Minkowski sum of~$\simplex_K$.~
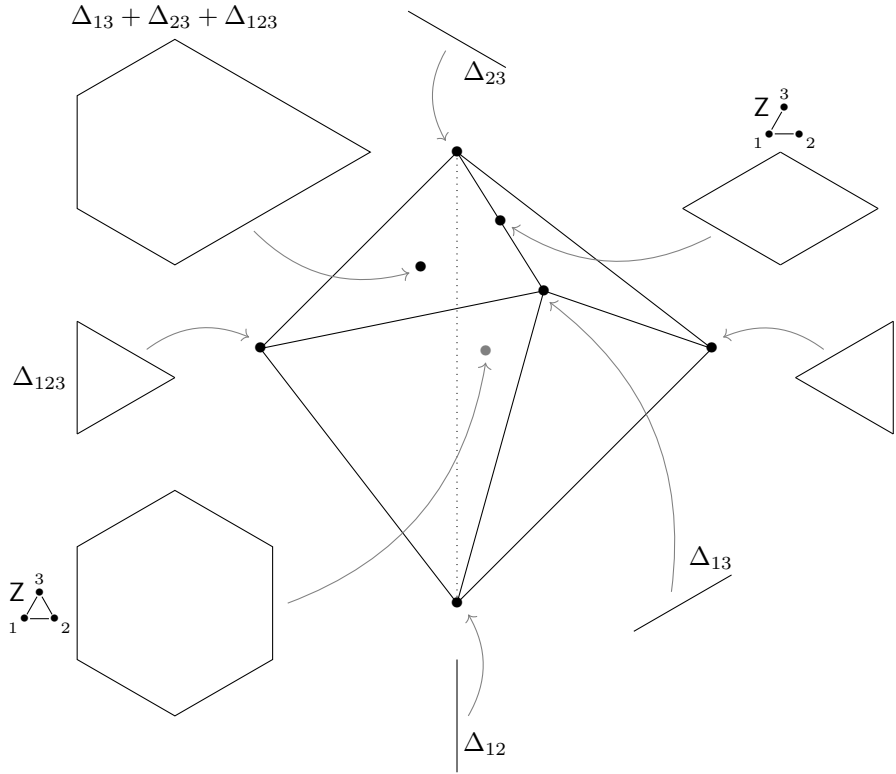
\begin{figure}
	\centerline{\begin{tikzpicture}[x={(1cm,0cm)}, y={(0cm,1cm)}, z={(3.85mm,3.85mm)}, scale=1.5]
	\newcommand{\translatepoint}[1]{\coordinate (trans) at (#1);}
	\newcommand{\cornerGraph}{
		\begin{tikzpicture}[inner sep=0]
			\node[label={[xshift=-.15cm, yshift=-.25cm] \tiny 1}] (1) at (-0.2,-0.35) {\tiny{$\bullet$}};
			\node[label={[xshift=.15cm, yshift=-.25cm] \tiny 2}] (2) at (0.2,-0.35) {\tiny{$\bullet$}};
			\node[label={[xshift=0cm, yshift=.05cm] \tiny 3}] (3) at (0,0) {\tiny{$\bullet$}};
			\draw (2) -- (1) -- (3);
		\end{tikzpicture}
	}
	\newcommand{\triangleGraph}{
		\begin{tikzpicture}[inner sep=0]
			\node[label={[xshift=-.15cm, yshift=-.25cm] \tiny 1}] (1) at (-0.2,-0.35) {\tiny{$\bullet$}};
			\node[label={[xshift=.15cm, yshift=-.25cm] \tiny 2}] (2) at (0.2,-0.35) {\tiny{$\bullet$}};
			\node[label={[xshift=0cm, yshift=.05cm] \tiny 3}] (3) at (0,0) {\tiny{$\bullet$}};
			\draw (1) -- (2) -- (3) -- (1);
		\end{tikzpicture}
	}
	\newdimen\Rad
	\Rad=1cm
	\coordinate (a) at (90:\Rad);
	\coordinate (b) at (150:\Rad);
	\coordinate (c) at (210:\Rad);
	\coordinate (d) at (270:\Rad);
	\coordinate (e) at (330:\Rad);
	\coordinate (f) at (390:\Rad);
	\coordinate (x) at (0,0,0);
	\coordinate (y) at (0,4,0);
	\coordinate (z) at (0,2,2);
	\coordinate (s) at (-2,2,.67);
	\coordinate (t) at (2,2,.67);
	\draw[dotted] (x)--(y);
	\draw (z)--(s);
	\draw (z)--(t);
	\draw (z)--(x);
	\draw (z)--(y);
	\draw (y)--(s);
	\draw (y)--(t);
	\draw (x)--(s);
	\draw (x)--(t);
	\draw (x) node {$\bullet$};
	\draw (y) node {$\bullet$};
	\draw (s) node {$\bullet$};
	\draw (t) node {$\bullet$};
	\draw (z) node {$\bullet$};
	\draw ($0.5*(y)+0.5*(z)$) node {$\bullet$};
	\draw ($0.33*(y)+0.33*(z)+0.33*(s)$) node {$\bullet$};
	\draw[gray] ($0.33*(x)+0.33*(y)+0.33*(z)$) node {$\bullet$};
	%
	\translatepoint{0,-1,0}
	\draw ($-0.5*(a) + (trans)$) -- ($0.5*(a) + (trans)$);
	\draw ($(trans) + (0.25,-0.25,0)$) node {$\Delta_{12}$};
	\draw[gray,->] ($(trans)+(0.1,0,0)$) to[bend right] ($(x)+(0.1,-0.1,0)$);
	%
	\translatepoint{0,5,0}
	\draw ($-0.5*(b) + (trans)$) -- ($0.5*(b) + (trans)$);
	\draw ($(trans) + (0.25,-0.3,0)$) node {$\Delta_{23}$};
	\draw[gray,->] ($(trans)+(-0.1,-0.1,0)$) to[bend right] ($(y)+(-0.1,0.1,0)$);
	%
	\translatepoint{2,0,0}
	\draw ($-0.5*(c) + (trans)$) -- ($0.5*(c) + (trans)$);
	\draw ($(trans) + (0.25,0.4,0)$) node {$\Delta_{13}$};
	\draw[gray,->] ($(trans)+(-0.1,0.1,0)$) to[bend right] ($(z)+(0.05,-0.1,0)$);
	%
	\translatepoint{3,2,0}
	\draw ($(0,0,0) + (trans)$) -- ($-1*(c) + (trans)$);
	\draw ($(0,0,0) + (trans)$) -- ($-1*(b) + (trans)$);
	\draw ($-1*(c) + (trans)$) -- ($-1*(b) + (trans)$);
	\draw[gray,->] ($(trans)+(0.25,0.25,0)$) to[bend right] ($(t)+(0.1,0.1,0)$);
	%
	\translatepoint{-2.5,2,0}
	\draw ($(0,0,0) + (trans)$) -- ($(c) + (trans)$);
	\draw ($(0,0,0) + (trans)$) -- ($(b) + (trans)$);
	\draw ($(c) + (trans)$) -- ($(b) + (trans)$);
	\draw ($(trans) + (-1.2,0,0)$) node {$\Delta_{123}$};
	\draw[gray,->] ($(trans)+(-0.25,0.25,0)$) to[bend left] ($(s)+(-0.1,0.1,0)$);
	%
	\translatepoint{2,3.5,0}
	\draw ($(0,0,0) + (trans)$) -- ($-1*(c) + (trans)$);
	\draw ($(0,0,0) + (trans)$) -- ($-1*(b) + (trans)$);
	\draw ($-1*(b) + (trans)$) -- ($-1*(b) - (c) + (trans)$);
	\draw ($-1*(c) + (trans)$) -- ($-1*(b) - (c) + (trans)$);
	\draw ($(trans) + (0.7,0.9,0)$) node {$\mathsf{Z}$};
	\draw ($(trans)+(.9,0.8,0)$) node {\cornerGraph};
	\draw[gray,->] ($(trans)+(0.25,-0.25,0)$) to[bend left] ($0.5*(y)+0.5*(z)+(0.1,-0.05,0)$);
	%
	\translatepoint{-2.5,4,0}
	\draw ($-1*(b) + (c) +(trans)$) -- ($-1*(b) - (c) + (trans)$) -- ($(b) - (c) + (trans)$) -- ($(b) + (trans)$) -- ($(c) + (trans)$) -- cycle;
	\draw ($(trans) + (0,1.2,0)$) node {$\Delta_{13}+\Delta_{23}+\Delta_{123}$};
	\draw[gray,->] ($(trans)+(0.7,-0.7,0)$) to[bend right] ($0.33*(y)+0.33*(z)+0.33*(s)+(-0.1,-0.05,0)$);
	%
	\translatepoint{-2.5,0,0}
	\draw ($(a) + (trans)$) -- ($(b) + (trans)$) -- ($(c) + (trans)$) -- ($(d) + (trans)$) -- ($(e) + (trans)$) -- ($(f) + (trans)$) -- cycle;
	\draw ($(trans) + (-1.4,0.1,0)$) node {$\mathsf{Z}$};
	\draw ($(trans) + (-1.2,0,0)$) node {\triangleGraph};
	\draw[gray,->] ($(trans)+(1,0,0)$) to[bend right] ($0.33*(x)+0.33*(y)+0.33*(z)+(0,-0.1,0)$);
\end{tikzpicture}}
	\caption{A $3$-dimensional affine section of the deformation cone $\deformationCone(\gZono[K_3])$ for the triangle~$K_3$. The deformations of $\gZono[K_3]$ corresponding to some of the points of $\deformationCone(\gZono[K_3])$ are depicted.}
	\label{fig:3CycleDeformationCone}
\end{figure}
\end{example}

\begin{example}
For a triangle-free graph~$G \eqdef (V,E)$, the deformation cone~$\deformationCone(\gZono)$ has dimension~$|V|+|E|$ and~$|E|$ facets.
As before, the lineality is $|V|$-dimensional, given by the space of translations in~$\R^{|V|}$.
Thus~$\deformationCone(\gZono)$ is simplicial.

For instance, for the $4$-cycle~$C_4$, the graphical zonotope~$\gZono[C_4]$ is the $3$-dimensional zonotope depicted in the bottom right of \cref{fig:4CycleDeformationCone}, which arises as the Minkowski sum of $4$ vectors in a hyperplane of~$\R^4$.
Its deformation cone $\deformationCone(\gZono[C_4])$ lives in the $16$-dimensional space~$\smash{\R^{2^{[4]}}}$, has dimension~$8$, a lineality space of dimension~$4$, and $4$ facets.
It admits as irredundant description the following $8$ equations and $4$ inequalities:
\begin{align*}
    \b{h}_\varnothing & = -\b{h}_{1234} &
    \b{h}_{12} + \b h_{14} & = \b{h}_{124} + \b{h}_{1} &
    \b{h}_{1} + \b h_{2} & \geq \b{h}_{12} + \b{h}_\varnothing
    \\
    \b{h}_1 + \b h_3 & = \b{h}_{13} + \b{h}_\varnothing &
    \b{h}_{12} + \b h_{23} & = \b{h}_{123} + \b{h}_{2} &
    \b{h}_{2} + \b h_{3} & \geq \b{h}_{23} + \b{h}_\varnothing
    \\
    \b{h}_2 + \b h_4 & = \b{h}_{24} + \b{h}_\varnothing &
    \b{h}_{23} + \b h_{34} & = \b{h}_{234} + \b{h}_{3} &
    \b{h}_{3} + \b h_{4} & \geq \b{h}_{34} + \b{h}_\varnothing
    \\
     \b{h}_{123} + \b h_{134} & = \b{h}_{1234} + \b{h}_{13} &
     \b{h}_{14} + \b h_{34} & = \b{h}_{134} + \b{h}_{4} &
     \b{h}_{1} + \b h_{4} & \geq \b{h}_{14} + \b{h}_\varnothing.
\end{align*}
After quotienting the lineality and intersecting with an affine hyperplane, we get the $3$-simplex illustrated in \cref{fig:4CycleDeformationCone}.
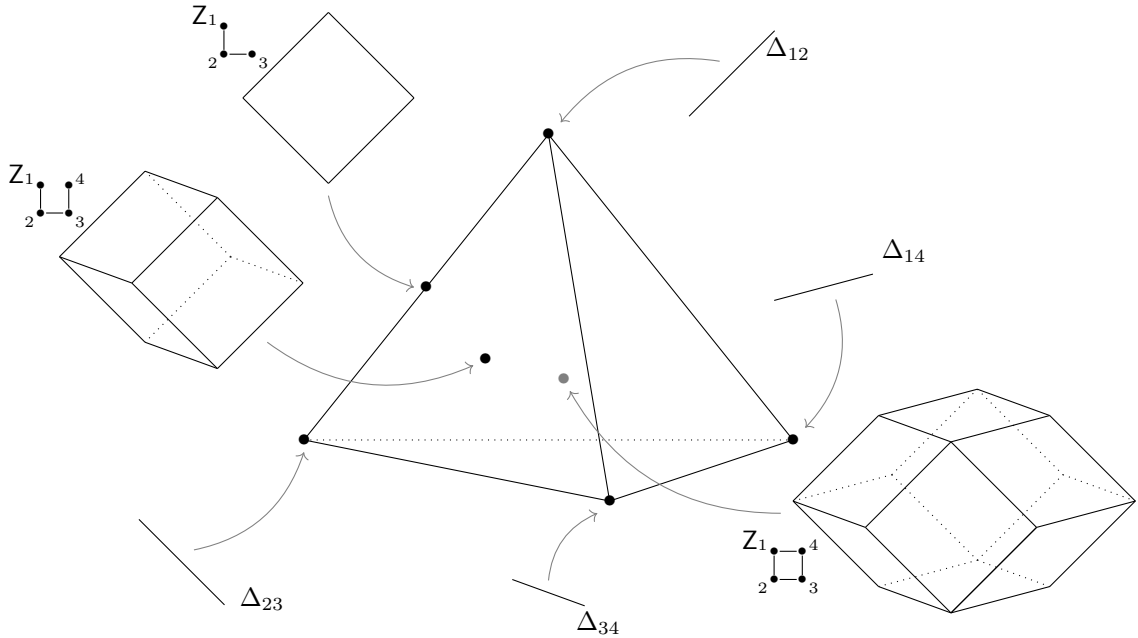
\begin{figure}
	\centerline{\begin{tikzpicture}[x={(1.25cm,0cm)}, y={(0cm,1.25cm)}, z={(-1.925mm, -3.85mm)}, scale=1.3]
	\newcommand{\translatepoint}[1]{\coordinate (trans) at (#1);}
	\newcommand{\angleGraph}{
		\begin{tikzpicture}[inner sep=0, scale=.15]
			\node[label={[xshift=-.15cm, yshift=-.05cm] \tiny 1}] (1) at (-1,1) {\tiny{$\bullet$}};
			\node[label={[xshift=-.15cm, yshift=-.25cm] \tiny 2}] (2) at (-1,-1) {\tiny{$\bullet$}};
			\node[label={[xshift=.15cm, yshift=-.25cm] \tiny 3}] (3) at (1,-1) {\tiny{$\bullet$}};
			\draw (1) -- (2) -- (3);
		\end{tikzpicture}
	}
	\newcommand{\horseshoe}{
		\begin{tikzpicture}[inner sep=0, scale=.15]
			\node[label={[xshift=-.15cm, yshift=-.05cm] \tiny 1}] (1) at (-1,1) {\tiny{$\bullet$}};
			\node[label={[xshift=-.15cm, yshift=-.25cm] \tiny 2}] (2) at (-1,-1) {\tiny{$\bullet$}};
			\node[label={[xshift=.15cm, yshift=-.25cm] \tiny 3}] (3) at (1,-1) {\tiny{$\bullet$}};
			\node[label={[xshift=.15cm, yshift=-.05cm] \tiny 4}] (4) at (1,1) {\tiny{$\bullet$}};
			\draw (1) -- (2) -- (3) -- (4);
		\end{tikzpicture}
	}
	\newcommand{\cycleGraph}{
		\begin{tikzpicture}[inner sep=0, scale=.15]
			\node[label={[xshift=-.15cm, yshift=-.05cm] \tiny 1}] (1) at (-1,1) {\tiny{$\bullet$}};
			\node[label={[xshift=-.15cm, yshift=-.25cm] \tiny 2}] (2) at (-1,-1) {\tiny{$\bullet$}};
			\node[label={[xshift=.15cm, yshift=-.25cm] \tiny 3}] (3) at (1,-1) {\tiny{$\bullet$}};
			\node[label={[xshift=.15cm, yshift=-.05cm] \tiny 4}] (4) at (1,1) {\tiny{$\bullet$}};
			\draw (1) -- (2) -- (3) -- (4) -- (1);
		\end{tikzpicture}
	}
	\coordinate (a) at (.7,.7,0);
	\coordinate (b) at (.7,-.7,0);
	\coordinate (c) at (.7,0,.7);
	\coordinate (d) at (.7,0,-.7);
	\coordinate (x) at (2.5,-0.5,0);
	\coordinate (y) at (2,2.5,0);
	\coordinate (z) at (4,0,0);
	\coordinate (t) at (0,0,0);
	\coordinate (x) at (2,2.5,0);
	\coordinate (y) at (0,0,0);
	\coordinate (z) at (2.5,-0.5,0);
	\coordinate (t) at (4,0,0);
	\draw (x) node {$\bullet$};
	\draw (y) node {$\bullet$};
	\draw (z) node {$\bullet$};
	\draw (t) node {$\bullet$};
	\draw (x) -- (y);
	\draw (x) -- (z);
	\draw (x) -- (t);
	\draw (y) -- (z);
	\draw[dotted]  (y) -- (t);
	\draw (z) -- (t);
	\draw ($0.5*(x) + 0.5*(y)$) node {$\bullet$};
	\draw ($0.33*(x) + 0.33*(y) + 0.33*(z)$) node {$\bullet$};
	\draw[gray] ($0.25*(x) + 0.25*(y) + 0.25*(z) + 0.25*(t)$) node {$\bullet$};
	%
	\translatepoint{3.5,3,0}
	\draw ($(trans) - 0.5*(a)$) -- ($(trans) + 0.5*(a)$);
	\draw ($(trans) + (.2,.2,0)$) node[right]{$\Delta_{12}$};
	\draw[gray,->] ($(trans) + (-0.1,0.1,0)$) to[bend right] ($(x) + (0.1,0.1,0)$);
	%
	\translatepoint{-1,-1,0}
	\draw ($(trans) - 0.5*(b)$) -- ($(trans) + 0.5*(b)$);
	\draw ($(trans) + (0.4,-.3,0)$) node[right]{$\Delta_{23}$};
	\draw[gray,->] ($(trans) + (0.1,0.1,0)$) to[bend right] ($(y) + (0,-0.1,0)$);
	%
	\translatepoint{2,-1.25,0}
	\draw ($(trans) - 0.5*(c)$) -- ($(trans) + 0.5*(c)$);
	\draw ($(trans) + (.15,-.25,0)$) node[right]{$\Delta_{34}$};
	\draw[gray,->] ($(trans) + (0,0.1,0)$) to[bend left] ($(z) + (-0.1,-0.1,0)$);
	%
	\translatepoint{4.25,1.25,0}
	\draw ($(trans) - 0.5*(d)$) -- ($(trans) + 0.5*(d)$);
	\draw ($(trans) + (0.4,.3,0)$) node[right]{$\Delta_{14}$};
	\draw[gray,->] ($(trans) + (0.1,-0.1,0)$) to[bend left] ($(t) + (0.1,0.1,0)$);
	%
	\translatepoint{-.5,2.8,0}
	\draw (trans) -- ($(trans) + (a)$);
	\draw (trans) -- ($(trans) + (b)$);
	\draw ($(trans) + (a)$) -- ($(trans) + (a) + (b)$);
	\draw ($(trans) + (b)$) -- ($(trans) + (a) + (b)$);
	\draw ($(trans) + (-0.1,0.5,0)$) node {$\mathsf{Z}_{\angleGraph}$};
	\draw[gray,->] ($(trans) + (0.7,-0.8,0)$) to[bend right] ($0.5*(x) + 0.5*(y) + (-0.1,0,0)$);
	%
	\translatepoint{-2,1.5,0}
	\draw (trans) -- ($(trans) + (a)$);
	\draw (trans) -- ($(trans) + (c)$);
	\draw ($(trans) + (a)$) -- ($(trans) + (a) + (c)$);
	\draw ($(trans) + (c)$) -- ($(trans) + (a) + (c)$);
	\draw (trans) -- ($(trans) + (b)$);
	\draw ($(trans) + (b)$) -- ($(trans) + (b) + (c)$);
	\draw ($(trans) + (c)$) -- ($(trans) + (b) + (c)$);
	\draw ($(trans) + (b) + (c)$) -- ($(trans) + (a) + (b) + (c)$);
	\draw ($(trans) + (a) + (c)$) -- ($(trans) + (a) + (b) + (c)$);
	\draw[dotted] ($(trans) + (a) + (b)$) -- ($(trans) + (a) + (b) + (c)$);
	\draw[dotted] ($(trans) + (a) + (b)$) -- ($(trans) + (b)$);
	\draw[dotted] ($(trans) + (a) + (b)$) -- ($(trans) + (a)$);
	\draw ($(trans) + (-.1,0.5,0)$) node {$\mathsf{Z}_{\horseshoe}$};
	\draw[gray,->] ($(trans) + (1.7,-0.7,0)$) to[bend right] ($0.33*(x) + 0.33*(y) + 0.33*(z) + (-0.1,-0.05,0)$);
	%
	\translatepoint{4,-0.5,0}
	\draw (trans) -- ($(trans) + (a)$);
	\draw (trans) -- ($(trans) + (b)$);
	\draw (trans) -- ($(trans) + (c)$);
	\draw[dotted] (trans) -- ($(trans) + (d)$);
	\draw ($(trans) + (a)$) -- ($(trans) + (a) + (c)$);
	\draw ($(trans) + (a)$) -- ($(trans) + (a) + (d)$);
	\draw ($(trans) + (b)$) -- ($(trans) + (b) + (c)$);
	\draw[dotted] ($(trans) + (b)$) -- ($(trans) + (b) + (d)$);
	\draw ($(trans) + (c)$) -- ($(trans) + (c) + (a)$);
	\draw ($(trans) + (c)$) -- ($(trans) + (c) + (b)$);
	\draw[dotted] ($(trans) + (d)$) -- ($(trans) + (d) + (a)$);
	\draw[dotted] ($(trans) + (d)$) -- ($(trans) + (d) + (b)$);
	\draw ($(trans) + (a) + (c)$) -- ($(trans) + (a) + (c) + (b)$);
	\draw ($(trans) + (a) + (c)$) -- ($(trans) + (a) + (c) + (d)$);
	\draw[dotted] ($(trans) + (a) + (d)$) -- ($(trans) + (a) + (d) + (b)$);
	\draw ($(trans) + (a) + (d)$) -- ($(trans) + (a) + (d) + (c)$);
	\draw ($(trans) + (b) + (c)$) -- ($(trans) + (b) + (c) + (a)$);
	\draw ($(trans) + (b) + (c)$) -- ($(trans) + (b) + (c) + (d)$);
	\draw[dotted] ($(trans) + (b) + (d)$) -- ($(trans) + (b) + (d) + (a)$);
	\draw[dotted] ($(trans) + (b) + (d)$) -- ($(trans) + (b) + (d) + (c)$);
	\draw ($(trans) + (b) + (c)$) -- ($(trans) + (b) + (c) + (a)$);
	\draw ($(trans) + (a) + (b) + (c)$) -- ($(trans) + (a) + (b) + (c) + (d)$);
	\draw[dotted] ($(trans) + (a) + (b) + (d)$) -- ($(trans) + (a) + (b) + (c) + (d)$);
	\draw ($(trans) + (a) + (c) + (d)$) -- ($(trans) + (a) + (b) + (c) + (d)$);
	\draw ($(trans) + (b) + (c) + (d)$) -- ($(trans) + (a) + (b) + (c) + (d)$);
	\draw ($(trans) + (-0.1,-0.5,0)$) node {$\mathsf{Z}_{\cycleGraph}$};
	\draw[gray,->] ($(trans) + (-.1,-.1,0)$) to[bend left] ($0.25*(x) + 0.25*(y) + 0.25*(z) + 0.25*(t) + (0.05,-0.1,0)$);
\end{tikzpicture}}
	\caption{A $3$-dimensional affine section of the deformation cone $\deformationCone(\gZono[C_4])$ for the $4$-cycle~$C_4$. The deformations of $\gZono[C_4]$ corresponding to some of the points of $\deformationCone(\gZono[C_4])$ are depicted.}
	\label{fig:4CycleDeformationCone}
\end{figure}
\end{example}


\subsection{Simplicial graphical deformation cones}
\label{subsec:triangleFree}

As an immediate corollary, we obtain a characterization of those graphical zonotopes whose deformation cone is simplicial. 

\begin{corollary}
\label{cor:simplicialGraphicalTypeCones}
The deformation cone~$\deformationCone(\gZono)$ is simplicial (modulo its lineality) if and only if $G$ is triangle-free.
\end{corollary}

\begin{proof}
If~$G$ is triangle-free, the deformation cone~$\deformationCone(\gZono)$ has dimension~$|V|+|E|$, lineality space of dimension~$|V|$, and~$|E|$ facets, and hence it is simplicial.
If~$G$ is not triangle-free, then we claim that the number of induced cliques~$K$ of~$G$ with~$|K| \ge 2$ is strictly less than the number of triples~$(u,v,S)$ with~$\{u,v\} \in E$ and~$S \subseteq N(u) \cap N(v)$.
Indeed, each induced clique~$K$ of~$G$ with~$|K| \ge 2$ already produces $\binom{|K|}{2}$ triples of the form~$(u, v, K \ssm \{u,v\})$ which satisfy~$\{u,v\} \in E$ and $K \ssm \{u,v\} \subseteq N(u) \cap N(v)$ and are all distinct.
Since $\binom{|K|}{2} > |K|$ as soon as $|K|\geq 3$, by \cref{coro:numerologyGraphicalTypeCone}, this shows that the deformation cone~$\deformationCone(\gZono)$ is not simplicial.
\end{proof}

\begin{corollary}
If $G$ is triangle-free, then every deformation of $\gZono$ is a zonotope, which is the graphical zonotope of a subgraph of~$G$ up to rescaling of the generators.
\end{corollary}

\begin{proof}
For any induced clique~$K$ of~$G$ of size at least~$2$, $\simplex_K$ is a Minkowski indecomposable ${(|K|-1)}$-dimensional polytope in the deformation cone~$\deformationCone(\gZono)$ (see for example \cite[15.1.3]{Gruenbaum2003} for a certificate of indecomposability). It spans therefore a ray of~$\deformationCone(\gZono)$.  When $G$ is triangle-free, 
the deformation cone modulo its lineality is of dimension~$|E|$, and the polytopes $\Delta_e$ for $e\in E$ account for the $|E|$ rays of the simplicial deformation cone~$\deformationCone(\gZono)$. 

Therefore, each polytope $P\in\deformationCone(\gZono)$ can be uniquely expressed as a Minkowski sum 
\[P=\sum_{e\in E}\lambda_e \Delta_e\]
 with nonnegative coefficients~$\lambda_e$. Since each $\Delta_e$ is a segment, $P$ is a zonotope, normally equivalent to the graphical zonotope of the subgraph $G'=(V,E')$ with $E'=\set{e\in E}{\lambda_e\neq 0}$.
\end{proof}

\section*{Acknowledgments}

We are grateful to Raman Sanyal for letting us know about his proof of \cref{thm:basisAndEquationsGraphicalTypeCones} with Josephine Yu, as well as for many helpful discussions on the topic.

\bibliographystyle{alpha}
\bibliography{GraphicalZonotopes}
\label{sec:biblio}

\end{document}